\documentclass[a4paper,12pt]{amsart}
\usepackage{amsmath,amssymb,amsfonts,amsthm}
\usepackage{latexsym,mathrsfs}
\usepackage{graphicx}
\usepackage[all]{xy}
\input xypic
\usepackage{color}
\usepackage[pagebackref]{hyperref}
\hypersetup{colorlinks=true,linkcolor=red,citecolor=blue}
\usepackage[OT2,T1]{fontenc}
\usepackage[utf8]{inputenc}
\usepackage{lmodern}
\usepackage{microtype}
\usepackage{enumerate}
\usepackage{tikz-cd}
\usepackage{bm}
\usepackage[normalem]{ulem}
\usepackage{mathtools}
\usepackage{tensor}
\theoremstyle{plain}
\newtheorem{theorem}[subsection]{{\bf Theorem}}

\newtheorem{corollary}[subsection]{{\bf Corollary}}
\newtheorem{proposition}[subsection]{{\bf Proposition}}
\newtheorem{lemma}[subsection]{{\bf Lemma}}
\theoremstyle{definition}
\newtheorem*{definition}{{\bf Definition}}
\theoremstyle{remark}
\newtheorem{remark}[subsection]{{\it Remark}}
\newtheorem{example}[subsection]{{\it Example}}
\numberwithin{equation}{subsection}

\DeclareMathOperator{\Aut}{Aut}

\DeclareMathOperator{\Hom}{Hom}

\DeclareMathOperator{\Sym}{Sym}
\DeclareMathOperator{\Fix}{Fix}
\DeclareMathOperator{\id}{id}

\DeclareBoldMathCommand{\bbot}{\bot}

\DeclareSymbolFont{cyrletters}{OT2}{wncyr}{m}{n}
\DeclareMathSymbol{\Sha}{\mathalpha}{cyrletters}{"58}

\DeclareMathOperator{\Pold}{Pol}

\DeclareMathOperator{\deff}{def}
\DeclareMathOperator{\Alt}{Alt}

\DeclareMathOperator{\mypartial}{\partial}

\newcommand{\llangle}{\langle\!\langle}
\newcommand{\rrangle}{\rangle\!\rangle}

\newcommand{\Z}{\mathbb{Z}}

\newcommand{\C}{\mathbb{C}}

\newcommand{\ncl}[1]{\langle\!\langle #1\rangle\!\rangle}

\begin{document}
\title[Polynomial stability and soficity]{Polynomial permutation stability, soficity, and universal polynomial groups}
\author{Primo\v z Moravec}
\address{{
 Faculty of  Mathematics and Physics, University of Ljubljana,
 and Institute of Mathematics, Physics and Mechanics,
 Slovenia}}
 \email{primoz.moravec@fmf.uni-lj.si}
\thanks{ORCID: \url{https://orcid.org/0000-0001-8594-0699}.}
\thanks{The author acknowledges the financial support from the Slovenian Research and Innovation Agency (ARIS), research core funding No. P1-0222, and project No. J1-50001.}
\date{\today}
\begin{abstract}
\noindent
We introduce polynomial permutation stability for maps between
groups, extending ordinary permutation stability from
homomorphisms to polynomial maps defined by iterated finite
differences. We show that a group $G$ is polynomially
permutation stable in degree $s$ if and only if its universal
polynomial group $\Pold_s(G)$ is permutation stable.

While every finite group is polynomially permutation stable
in degree two, we prove that the cyclic group of order six
is not stable in degree three. More precisely, we construct
a finite-index direct-product subgroup of $\Pold_3(C_6)$
to which an instability theorem of Ioana applies.

We also show that polynomial sofic approximations characterize
the usual class of sofic groups. In degree two, we obtain a
split exact sequence relating $\Pold_2(G)$ to a restricted
wreath product, with applications to soficity and property
$(T)$. A quotient-transfer principle yields infinite,
finitely presented, nonperfect groups that are polynomially
permutation stable in every degree. Finally, we obtain
nonamenability and largeness results for higher-degree
universal polynomial groups of finite groups.
\end{abstract}
\maketitle


\section{Introduction}
\label{s:intro}

\noindent
Polynomial maps between groups generalize homomorphisms by
requiring sufficiently high iterated finite differences to
vanish. They were systematically studied by Leibman
\cite{Lei02} and arise naturally in higher-order Fourier
analysis and in the study of Gowers norms
\cite{GT12,GTZ12}. More recently, Jamneshan and Thom
\cite{JT24} developed a theory of polynomial maps between
nonabelian groups, including a description of universal
quadratic maps and their associated groups.

The aim of this paper is to investigate polynomial maps
from the perspective of permutation stability and sofic
approximation. Our main motivation is the question of
whether approximately polynomial maps into finite symmetric
groups can be approximated by genuine polynomial maps of
the same degree.

Recall that a group is permutation stable ($P$-stable)
if every almost-homomorphism into finite symmetric groups,
equipped with normalized Hamming distance, is close to a
sequence of genuine homomorphisms. This notion originates
in the work of Glebsky and Rivera \cite{GR09} and has been
studied extensively; see, for instance,
\cite{AP15,BLT19,AP24}. We introduce an analogous notion
for polynomial maps of degree at most $s$, replacing
asymptotic multiplicativity by asymptotic vanishing of
$(s+1)$-fold differences. A group satisfying the resulting
stability property is called $P_s$-stable. For $s=1$,
this is precisely ordinary permutation stability.

Our approach uses the universal polynomial group
$\Pold_s(G)$ constructed by Jamneshan and Thom
\cite{JT24}. This group represents unital polynomial
maps of degree at most $s$ from $G$ into arbitrary
groups. The universal property allows us to translate
polynomial permutation stability into ordinary permutation
stability.

\begin{theorem}
\label{thm:Pscharacterization}
Let $G$ be a group and let $s\geq1$. Then $G$ is
$P_s$-stable if and only if $\Pold_s(G)$ is $P$-stable.
\end{theorem}

This equivalence makes the existing theory of permutation
stability available in the polynomial setting. In
particular, every finite group is $P_2$-stable: the
quadratic universal group $\Pold_2(G)$ is finite whenever
$G$ is finite, by Jamneshan--Thom \cite{JT24}, and
finite groups are $P$-stable by Glebsky--Rivera
\cite{GR09}. It is therefore natural to ask whether
every finite group is $P_s$-stable for higher degrees.

Our main result shows that this fails already in
degree three.

\begin{theorem}
\label{thm:C6}
The cyclic group $\Z/6\Z$ is not $P_3$-stable.
\end{theorem}

In fact, we prove the stronger assertion that
$\Pold_3(\Z/6\Z)$ is not very flexibly $P$-stable.
The proof is based on Ioana's instability theorem for
direct products \cite{Ioa20}. We construct a subgroup
of index $2592$ in $\Pold_3(\Z/6\Z)$ which is
isomorphic to a direct product $\Sigma\times\Lambda$,
where
\[
\Sigma\twoheadrightarrow F_3,
\qquad
\Lambda\twoheadrightarrow\Z.
\]
Here $F_3$ denotes the free group of rank three.
The construction involves a presentation for a finite
extension of $\Pold_3(\Z/m\Z)$, a direct-product
decomposition arising from coprime orders, and two
different geometric and algebraic ingredients.
A genus-three surface subgroup yields the nonabelian
free quotient, while an affine representation over
the Eisenstein integers produces the infinite cyclic
quotient.

The result is related to recent work of Alekseev
and Thom \cite{AT26}, who showed that
$\Pold_3(\Z/3\Z)$ is infinite and nonamenable.
Their result demonstrates that universal polynomial
groups of finite cyclic groups can have complicated
structure in degree three. Our construction provides
a different type of information: a finite-index
direct-product subgroup which allows us to apply
an obstruction to permutation stability.

We also investigate polynomial analogues of sofic
approximations. A polynomial sofic approximation of
degree at most $s$ is a separating almost-polynomial
sequence of that degree. Since polynomial maps
generalize homomorphisms, such approximations might
appear to define a larger class of groups than the
usual class of sofic groups. We prove that this is
not the case.

\begin{theorem}
\label{thm:soficcharacterization}
Let $G$ be a countable group and let $s\geq1$.
Then $G$ admits a polynomial sofic approximation
of degree at most $s$ if and only if $G$ is sofic.
\end{theorem}

The main ingredient is a degree-reduction argument.
We show that the restriction of a polynomial map
of degree at most $s$ to the derived subgroup has
degree at most $s-1$. An analogous statement holds
for almost-polynomials. Combined with the fact that
extensions of sofic groups by amenable groups are
sofic, this reduces the problem inductively to
ordinary sofic approximations.

Although polynomial sofic approximations do not
enlarge the class of sofic groups, the soficity of
the universal polynomial groups is a separate
question. We show that $\Pold_s(G)$ being sofic
implies that $G$ is sofic. In degree two the converse
also holds, while for $s\geq3$ the converse remains
open.

We further establish a stability-transfer principle:
if $K$ is a normal subgroup of a countable group $G$
and $K$ has no nontrivial sofic quotient, then
$G$ is $P_s$-stable if and only if $G/K$ is
$P_s$-stable. This gives examples of infinite,
nonperfect groups that are polynomially permutation
stable in every degree. In particular, if $N$ is
a countable simple nonsofic group, both
\[
N\times\Z/2\Z
\qquad\text{and}\qquad
N*\Z/2\Z
\]
are $P_s$-stable for all $s\geq1$.
Using recent results on the nonsoficity of
higher-dimensional Brin--Thompson groups
\cite{SSFF}, we obtain finitely presented examples
of this kind.

The final part of the paper is concerned with the
structure of universal polynomial groups.
In degree two, we use the description of Jamneshan
and Thom to establish a split exact sequence
\[
1\longrightarrow\Pold_2(G)
\longrightarrow G^{\rm ab}\wr G
\longrightarrow G^{\rm ab}
\longrightarrow1,
\]
where $\wr$ denotes the restricted wreath product.
Together with the soficity of wreath products
\cite{HS18}, this provides a characterization
of the soficity of $\Pold_2(G)$ in terms of $G$.
We also show that, for an infinite group $G$
and $s\geq2$, the group $\Pold_s(G)$ has
property $(T)$ if and only if $G$ is perfect
and has property $(T)$.

For finite groups, we study nonamenability and
largeness of higher-degree universal polynomial
groups. The computation of $\Pold_3(\Z/3\Z)$
in \cite{AT26}, together with functoriality,
implies that $\Pold_s(G)$ is nonamenable whenever
$G$ is finite, $s\geq3$, and
$3$ divides $|G^{\rm ab}|$.
We complement this observation with finite groups
whose abelianization has order two. More precisely,
for a nontrivial finite perfect group $N$ and an
automorphism $\alpha$ satisfying $\alpha^2=1$,
we obtain a relative presentation of
\[
\Pold_s(N\rtimes_\alpha\Z/2\Z).
\]
Using results on adjoining high-powered relations
\cite{Lac07} and estimates based on
$2$-deficiency \cite{BT11}, we show that these
groups are large for all sufficiently large $s$
and give an explicit sufficient bound.

The paper is organized as follows.
Sections~\ref{s:prelim} and
\ref{s:almostpolynomials} introduce the necessary
background and basic properties of almost-polynomials.
Section~\ref{s:polysoficapprox} studies polynomial
sofic approximations, while
Section~\ref{s:polystability} develops polynomial
permutation stability and its basic consequences.
The proof of the cubic instability theorem occupies
Section~\ref{s:C6}. Finally,
Section~\ref{s:Pold} studies structural properties
of universal polynomial groups, including soficity,
property $(T)$, and largeness.


\section{Preliminaries}
\label{s:prelim}

\subsection{Polynomial maps}
\label{ss:poly_prelim}

For groups $G$ and $H$, let $H^G$ be the set of all maps
$f:G\to H$. Given $g\in G$, we define the left difference operator $\mypartial_g$ on $H^G$ by the rule
$$(\mypartial_g f)(x)=f(gx)f(x)^{-1}.$$
Given $g_1,\ldots ,g_k\in G$, we use the notation
$$\mypartial_{g_1,\ldots,g_k}f=\mypartial_{g_1}\circ\cdots\circ\mypartial_{g_k}f.$$
According Leibman \cite{Lei02}, the map $f$ is said to be \emph{polynomial of degree at most $s$} if
$$(\mypartial_{g_1,\ldots ,g_{s+1}}f)(x)=1$$
for all $x,g_1,\ldots ,g_{s+1}\in G$. 

Observe that
$$(\mypartial_gf)(x)
=f(gx)f(1)^{-1}f(1)f(x)^{-1}
=(\mypartial_{gx}f)(1)\cdot(\mypartial_xf)(1)^{-1}$$ implies
\begin{align*}
    (\mypartial_{g_1,\ldots ,g_{s+1}}f)(x) &=
    (\mypartial_{g_1}(\mypartial_{g_2,\ldots,g_{s+1}}f))(x)\\
    &= (\mypartial_{g_1x}(\mypartial_{g_2,\ldots,g_{s+1}}f))(1)
    \cdot (\mypartial_x(\mypartial_{g_2,\ldots,g_{s+1}}f))(1)^{-1}\\
    &= (\mypartial_{g_1x,g_2,\ldots,g_{s+1}}f)(1)
    \cdot (\mypartial_{x,g_2,\ldots,g_{s+1}}f)(1)^{-1}.
\end{align*}
This shows that $f$ is polynomial of degree $\le s$ if and only if 
$$(\mypartial_{g_1,\ldots ,g_{s+1}}f)(1)=1$$
for all $g_1,\ldots ,g_{s+1}\in G$. This aligns with the definition of polynomial maps used by Jamneshan and Thom \cite{JT24}. 

\subsection{Universal polynomial construction}
\label{ss:universal_prelim}

The material of this subsection is adapted from \cite{JT24}. Let $G$ be a group. Let $F_G$ be the free group on formal generators $\{ x_g\mid g\in G\setminus\{1\}\}$, and set $x_1=1$ by notation. Define a map $\iota :G\to F_G$ by $\iota(g)=x_g$.
Let $R_s(G)$ be the normal closure in $F_G$ of the set of all words $(\mypartial_{g_1,\ldots ,g_{s+1}}\iota)(1)$, where $g_1,\ldots ,g_{s+1}$ run through $G$. Put
$\Pold_s(G)=F_G/R_s(G)$.

\begin{proposition}[\cite{JT24}, Theorem 1.2, Remark 2.5]
    \label{prop:Pold}
    There exists a universal unital polynomial map $\iota_s:G\to\Pold_s(G)$ of degree $\le s$ such that, 
    for every unital polynomial map $f:G\to H$ of degree $\le s$, there exists a unique homomorphism $\psi_f: \Pold_s(G)\to H$ such that $f=\psi_f\circ \iota_s$.
\end{proposition}

Explicitly, if $q:F_G\to\Pold_s(G)$ is the quotient homomorphism, then $\iota_s=q\circ\iota:G\to\Pold_s(G)$ is the universal polynomial map mentioned in Proposition \ref{prop:Pold}. 
Note that $\iota_s$ is injective. Namely, the identity map $\id_G:G\to G$ is a unital polynomial of degree $\le s$, therefore $\psi_{\id_G}\circ\iota_s=\id_G$.

\subsection{Almost-homomorphisms}
\label{ss:almosthom_prelim}
 
This is mainly taken from Becker, Lubotzky, Thom \cite{BLT19}.

Equip the symmetric groups $\Sym(X)$ with the normalized Hamming distance:
$$d_X(\sigma,\tau)=1-\frac{1}{|X|}|\Fix(\sigma^{-1}\tau)|.$$
The metric $d_n$ is bi-invariant.

Let $G$ be a countable group. 
An \emph{almost-homomorphism} of $G$ is a sequence of maps 
$f_n:G\to\Sym (X_n)$ satisfying
$$\lim_{n\to\infty} d_{X_n}(f_n(g)f_n(h),f_n(gh))=0$$
for all $g,h\in G$. Note that almost-homomorphisms are \emph{asymptotically unital}, that is,
$$\lim_{n\to\infty} d_{X_n}(f_n(1),1)=0.$$

Let $(f_n)_n$ be an almost-homomorphism of $G$. Choose a nonprincipal ultrafilter $\mathcal{U}$ on $\mathbb{N}$, and form the metric ultraproduct \cite{Gol22}
$$\mathfrak{S}_{\mathcal{U}}=\prod_{\mathcal{U}} (\Sym(X_n),d_{X_n})=\frac{\prod_n \Sym(X_n)}{\left\{(\sigma_n)_n\mid \lim_{n\to\mathcal{U}}d_{X_n}(\sigma_n,1)=0\right\}}.$$
Define $f:G\to\mathfrak{S}_{\mathcal{U}}$ by
$$f(g)=[f_n(g)]_{\mathcal{U}}.$$
The fact that $(f_n)_n$ is an almost-homomorphism implies that $f$ is a homomorphism of groups. 

Conversely, 
let $G$ be a countable group and let
$\Theta:G\to\mathfrak{S}_\mathcal{U}$
be a homomorphism into a metric ultraproduct.  For each $g\in G$,
choose representatives
$f_n(g)\in\Sym(X_n)$
such that
$$\Theta(g)=[(f_n(g))_n]_{\mathcal U}.$$
Since $\Theta$ is a homomorphism, for every $g,h\in G$ one has
$$
\lim_{n\to\mathcal U}
d_{X_n}\bigl(f_n(gh),f_n(g)f_n(h)\bigr)=0.
$$
As $G$ is countable, a standard diagonal argument yields a subsequence $(n_k)$ such that
$$\lim_{k\to\infty}d_{X_{n_k}}\bigl(
f_{n_k}(gh),f_{n_k}(g)f_{n_k}(h)
\bigr) =0
$$
for every $g,h\in G$.  Thus, after passing to a subsequence, every
homomorphism of a countable group into a metric ultraproduct of
symmetric groups is represented by an almost homomorphism in the usual
sequential sense.  Moreover, any countable collection of additional
ultralimit conditions may be preserved in the same diagonalization. In
particular, if
$d_{\mathcal U}\bigl(\Theta(g),1\bigr)=1$ for all $g\neq 1$,
then the subsequence may be chosen so that
$$\lim_{k\to\infty}d_{X_{n_k}}\bigl(f_{n_k}(g),1\bigr) =1$$
for every $g\neq 1$.


\subsection{Sofic groups}
\label{ss:sofic_prelim}

Let $G$ be a countable group. A \emph{sofic approximation} of $G$ is a sequence of maps $f_n:G\to\Sym(X_n)$ that satisfy
$$\lim_{n\to\infty}d_{X_n}(f_n(gh),f_n(g)f_n(h))=0$$
for all $g,h\in G$, and 
$$\lim_{n\to\infty}d_{X_n}(f_n(g),id)=1$$
for every $g\neq 1$. In other words, a sofic approximation of $G$ is a separating almost-homomorphism of $G$.

The group $G$ is said to be \emph{sofic} if it admits a sofic approximation. Equivalently, $G$ is sofic if and only if there exist finite sets $X_n$, a nonprincipal ultrafilter
$\mathcal U$ on $\mathbb N$, and a homomorphism
$\Theta:G\to\mathfrak{S}_\mathcal{U}$
such that
$d_{\mathcal U}\bigl(\Theta(g),1\bigr)=1$
for every $g\neq1$.

\subsection{Permutation stability}
\label{ss:pstab_prelim}

The almost-homomorphism $(f_n)_n$ of $G$ is \emph{close to a homomorphism} if there exists a sequence of group homomorphisms $\rho_n:G\to \Sym(X_n)$ such that
$$\lim_{n\to \infty} d_{X_n}(\rho_n(g),f_n(g))=0$$
for all $g\in G$. A group $G$ is \emph{permutation stable} (\emph{$P$-stable} for short) if every almost-homomorphism of $G$ is close to a homomorphism. Equivalently, $G$ is $P$-stable if and only if every homomorphism $f:G\to\mathfrak{S}_\mathcal{U}$ is liftable, i.e., there exists a homomorphism $\phi:G\to\prod_n\Sym(X_n)$ such that $f=Q\circ \phi$, where $Q:\prod_n\Sym(X_n)\to\mathfrak{S}_\mathcal{U}$ is the canonical projection, see \cite[Theorem 4.2]{AP15}.

The following is well-known, we include a proof for completeness:

\begin{lemma}
\label{lem:Pstabrel}
Let $G=F/N$, $N=\langle\!\langle R\rangle\!\rangle$,
where $F$ is a free group  and $R\subseteq F$ is an arbitrary set of relators. Suppose that $G$ is $P$-stable.
Let $f_n\colon F\to\Sym(X_n)$
be homomorphisms such that
$$d_{X_n}\bigl(f_n(r),1\bigr)\longrightarrow 0$$
for every $r\in R$.
Then there exist homomorphisms
$\rho_n\colon G\to\Sym(X_n)$
such that
$$d_{X_n}\bigl(f_n(w),\rho_n(wN)\bigr)\longrightarrow 0$$
for every $w\in F$.
Equivalently, if $\pi\colon F\to G$ is the quotient map, then
$$d_{X_n}\bigl(f_n(w),(\rho_n\circ\pi)(w)\bigr)\longrightarrow0$$
for every $w\in F$.
\end{lemma}

\begin{proof}
Choose a set-theoretic section
$s\colon G\to F$
of the quotient map $\pi\colon F\to G$, and define
$$\varphi_n(g)=f_n(s(g)).$$
Fix $g,h\in G$. Since
$s(g)s(h)s(gh)^{-1}\in N=\langle\!\langle R\rangle\!\rangle$,
there exist $r_1,\ldots,r_k\in R$, $\varepsilon_i\in\{-1,1\}$, and
$u_1,\ldots,u_k\in F$ such that
$$s(g)s(h)s(gh)^{-1}=
\prod_{i=1}^k u_i r_i^{\varepsilon_i}u_i^{-1}.$$
Using that $f_n$ is a homomorphism and that the Hamming metric is
bi-invariant, we obtain
$$
\begin{aligned}
d_{X_n}\bigl(\varphi_n(g)\varphi_n(h),\varphi_n(gh)\bigr)
&=
d_{X_n}\bigl(f_n(s(g)s(h)s(gh)^{-1}),1\bigr)\\
&\leq
\sum_{i=1}^k d_{X_n}\bigl(f_n(r_i),1\bigr)
\longrightarrow0.
\end{aligned}
$$
Hence $(\varphi_n)_n$ is an almost-homomorphism of $G$.
Since $G$ is $P$-stable, there exist homomorphisms
$$\rho_n\colon G\to\Sym(X_n)$$
such that
$$d_{X_n}\bigl(\varphi_n(g),\rho_n(g)\bigr)\longrightarrow0$$
for every $g\in G$.

Now fix $w\in F$. Since
$w\,s(\pi(w))^{-1}\in N$,
the same argument gives
$$d_{X_n}\bigl(f_n(w),f_n(s(\pi(w)))\bigr)\longrightarrow0.$$
Therefore
$$
\begin{aligned}
d_{X_n}\bigl(f_n(w),\rho_n(\pi(w))\bigr)
&\leq
d_{X_n}\bigl(f_n(w),f_n(s(\pi(w)))\bigr)\\
&\quad+
d_{X_n}\bigl(\varphi_n(\pi(w)),\rho_n(\pi(w))\bigr)
\longrightarrow0.
\end{aligned}
$$
This proves the claim.
\end{proof}


\section{Almost-polynomials}
\label{s:almostpolynomials}

\subsection{Basic notions}
\label{ss:almostbasic
}
\noindent

\begin{definition}[Almost-polynomial]
    A sequence of 
    asymptotically unital 
    maps $f_n:G\to\Sym(X_n)$ is \emph{almost-polynomial of degree $\le s$} if
    $$\lim_{n\to\infty} d_{X_n}((\mypartial_{g_1,\ldots ,g_{s+1}}f_n)(1),1)=0$$
    for all $g_1,\ldots ,g_{s+1}\in G$.
\end{definition}

Almost-polynomials are closely related to uniform $\varepsilon$-polynomials, \cite{JT24}. Let $(H,d)$ be a group with bi-invariant metric. A unital map $\phi:G\to H$ is a \emph{uniform $\varepsilon$-polynomial of degree $\le s$} if
    $$d(\mypartial_{g_1,...,g_{s+1}}\phi(1),1)\le\varepsilon$$
    for all $g_1,\ldots ,g_{s+1}\in G$.  Note, for example, that if $G$ is a finite group and $(f_n)_n$ is a unital almost-polynomial of $G$, then, given $\varepsilon>0$, all but possibly finitely many $f_n$ are uniform $\varepsilon$-polynomials $G\to\Sym(X_n)$.

\begin{example}
    \label{ex:deg1}
    Almost-polynomials of degree $1$ are precisely almost-homomorphisms. Assume first the sequence of maps $f_n:G\to\Sym(X_n)$ is an almost-homomorphism. Then $(f_n)_n$ is clearly asymptotically unital and
    \begin{align*}
        d_{X_n}((\mypartial_{h,g}f_n)(1),1) &=
        d_{X_n}(f_n(gh)f_n(h)^{-1}f_n(1)f_n(g)^{-1},1)\\
        &= d_{X_n}(f_n(gh)f_n(h)^{-1}f_n(g)^{-1},f_n(g)f_n(1)^{-1}f_n(g)^{-1})\\
        &\le d_{X_n}(f_n(gh)f_n(h)^{-1}f_n(g)^{-1},1)\\
        &\quad +
        d_{X_n}(1, f_n(g)f_n(1)^{-1}f_n(g)^{-1})\\
        &= d_{X_n}(f_n(gh),f_n(g)f_n(h))+d_{X_n}(f_n(1),1),
    \end{align*}
    and the latter converges to $0$ for all $g,h\in G$. 

    Conversely, let the maps $f_n:G\to\Sym(X_n)$ form an almost-polynomial of $G$ of degree $\le 1$. Then
    \begin{align*}
        d_{X_n}(f_n(gh),f_n(g)f_n(h)) 
        &= d_{X_n}((\mypartial_{h,g}f_n)(1),f_n(g)f_n(1)f_n(g)^{-1})\\
        &\le d_{X_n}((\mypartial_{h,g}f_n)(1),1)\\
        &\quad +d_{X_n}(1,f_n(g)f_n(1)f_n(g)^{-1})\\
        &=d_{X_n}((\mypartial_{h,g}f_n)(1),1)+d_{X_n}(1,f_n(1)),
    \end{align*}
    and the latter converges to $0$ for all $g,h\in G$.
\end{example}

\begin{proposition}
    \label{prop:degreeup}
    Let a sequence of maps $f_n:G\to\Sym(X_n)$ be an almost-polynomial of degree $\le s$. Then it is also almost-polynomial of degree $\le s+1$.
\end{proposition}

\begin{proof}
    Let $g_1,\ldots ,g_{s+1},g_{s+2}\in G$. Let $\phi_n=\mypartial_{g_2,\ldots ,g_{s+2}}f_n$. The map $\phi_n$ is asymptotically unital. We have 
    \begin{align*}
        d_{X_n}((\mypartial_{g_1,\ldots ,g_{s+2}}f_n)(1),1)
        &= 
        d_{X_n}((\mypartial_{g_1}\phi_n)(1),1)\\
        &= d_{X_n}(\phi_n(g_1)\phi_n(1)^{-1},1)\\
        &\le d_{X_n}(\phi_n(g_1),1)+d_{X_n}(\phi_n(1),1).
    \end{align*}
    The second summand of the latter sum tends to $0$. Also note that (Subsection \ref{ss:poly_prelim})
    \begin{align*}
        \phi_n(g_1) &= (\mypartial_{g_2,\ldots ,g_{s+2}}f_n)(g_1)\\
        &= (\mypartial_{g_2g_1,g_3,\ldots ,g_{s+2}}f_n)(1)
        (\mypartial_{g_1,g_3,\ldots ,g_{s+2}}f_n)(1)^{-1}
    \end{align*}
    and so the first summand goes to $0$ as well.
\end{proof}

\subsection{Degree reduction}
\label{ss:degreduction}

\noindent
The preceding results show that almost-polynomials
share some elementary properties with genuine
polynomial maps. A more useful structural phenomenon
is that the degree decreases upon restriction to
the derived subgroup. This observation will be
the main ingredient in our characterization of
polynomial sofic approximations.

\begin{lemma}
    \label{lem:lowerdeg}
    Let $f:G\to H$ be a unital polynomial of degree $\le s$, $s\ge 2$. Then there exists a normal subgroup $K$ of $G$ such that $G/K$ is abelian and $f\big|_K$ is a unital polynomial of degree $\le s-1$.
\end{lemma}

\begin{proof}
    Fix $\mathbf{a}=(a_2,\ldots ,a_s)\in G^{s-1}$. Denote $q_{\mathbf{a}}=\mypartial_{a_2,\ldots ,a_s}f$. As $f$ is polynomial of degree $\le s$, the map
    $\lambda_{\mathbf{a}}(x)=q_{\mathbf{a}}(x)q_{\mathbf{a}}(1)^{-1}$ is a homomorphism \cite[Lemma 2.1]{JT24}. We claim that the range of $\lambda_{\mathbf{a}}$ is abelian. Denote $\mu=\mypartial_{a_3,\ldots ,a_s}f$, and note that $q_{\mathbf{a}}=\mypartial_{a_2}\mu$. Setting $\mu^*(x)=\mu(x)\mu(1)^{-1}$, we have that $\mu^*$ is a unital quadratic map, and $q_{\mathbf{a}}=\mypartial_{a_2}\mu=\mypartial_{a_2}\mu^*$. Fix $k\in G$ and let $\beta_k(x)=\mu^*(k)^{-1}\mypartial_k\mu^*(x)$. By \cite[Proposition 2.4]{JT24}, the group $\beta_k(G)$ is abelian for every $k\in G$. Observe that
    $$\beta_{a_2}(x)=\mu^*(a_2)^{-1}q_{\mathbf{a}}(x)$$
    and 
    $$q_{\mathbf{a}}(1)=\mypartial_{a_2}\mu^*(1)=\mu^*(a_2)\mu^*(1)^{-1}=\mu^*(a_2).$$ 
    These two yield
    $$\lambda_{\mathbf{a}}(x)=\mu^*(a_2)\beta_{a_2}(x)\mu^*(a_2)^{-1},$$
    hence $\lambda_{\mathbf{a}}(G)=\mu^*(a_2)\beta_{a_2}(G)\mu^*(a_2)^{-1}$
    is abelian.

    Set
    $$K=\bigcap_{\mathbf{a}\in G^{s-1}}\ker\lambda_{\mathbf{a}}.$$
    The diagonal map
    $$G\to \prod_{\mathbf{a}\in G^{s-1}}\lambda_{\mathbf{a}}(G)$$
    has kernel $K$ and abelian target. This proves that $G/K$ is abelian.

    Pick $k_1,\ldots ,k_s\in K$. Then
    \begin{align*}
        \mypartial_{k_1,\ldots ,k_s}f(1) &= \mypartial_{k_1}q_{(k_2,\ldots ,k_s)}(1)\\
        &= q_{(k_2,\ldots ,k_s)}(k_1)q_{(k_2,\ldots ,k_s)}(1)^{-1}\\
        &= \lambda_{(k_2,\ldots ,k_s)}(k_1)\\
        &= 1,
    \end{align*}
    where the last step follows from the fact that $k_1\in K\subset\ker \lambda_{(k_2,\ldots ,k_s)}$.
\end{proof}

\begin{corollary}
    \label{cor:degreelowalmost}
    Let $(f_n)_n$ be an almost-polynomial of $G$ of degree $\le s$, where $s\ge 2$. Then $(f_n\big|_{G'})_n$ is an almost-polynomial of $G'$ of degree $\le s-1$.
\end{corollary}

\begin{proof}
    Let $(f_n)_n$ be an almost-polynomial of $G$ of degree $\le s$, $s\ge 2$. Pick an arbitrary nonprincipal ultrafilter $\mathcal{U}$, and form the metric ultraproduct
    $$\mathfrak{S}_{\mathcal{U}}=\prod_{\mathcal{U}} (\Sym(X_n),d_{X_n}),$$
    Define $f:G\to\mathfrak{S}_{\mathcal{U}}$ by
    $$f(g)=[f_n(g)]_{\mathcal{U}}.$$ By definition of $\mathfrak{S}_{\mathcal{U}}$, the asymptotic identities 
    $$d_{X_n}((\mypartial_{g_1,\ldots ,g_{s+1}}f_n)(1),1)\to 0$$
     yield
    $$\mypartial_{g_1,\ldots ,g_{s+1}}f(1)=1$$
    for all $g_1,\ldots ,g_{s+1}\in G$. Therefore $f$ is a genuine unital polynomial of degree $\le s$.
     By Lemma \ref{lem:lowerdeg}, the restriction of $f$ to $G'$ is a polynomial of degree $\le s-1$, that is,
    $$\lim_{n\to\mathcal{U}}d_{X_n}
    ((\partial_{c_1,\dots,c_s}f_n)(1),1)=0$$
    for all $c_1,\ldots ,c_s\in G'$.
    As this is true for every nonprincipal ultrafilter $\mathcal{U}$, we conclude that
    $$\lim_{n\to\infty}d_{X_n}
    ((\partial_{c_1,\dots,c_s}f_n)(1),1)=0,$$
    and the result follows.
\end{proof}

The next consequence can be compared with \cite[Theorem 1.7]{JT24}:
\begin{corollary}
    \label{cor:perfect}
    Let $G$ be a perfect group and let $(f_n)_n$ be an almost-polynomial of $G$. Then $(f_n)_n$ is an almost-homomorphism of $G$.
\end{corollary}


\section{Polynomially sofic approximations}
\label{s:polysoficapprox}

\noindent
We now consider the analogue of sofic approximation
obtained by replacing asymptotic multiplicativity
with asymptotic polynomiality. The separating
condition remains unchanged. Since every
almost-homomorphism is an almost-polynomial of
any higher degree, every sofic group admits such
approximations. The question is whether the
converse continues to hold.

\begin{definition}[Polynomially sofic approximation]
    Let $G$ be a countable group.
    The maps $f_n:G\to\Sym(X_n)$ form a \emph{polynomial sofic approximation of $G$ of degree $\le s$} if $(f_n)_n$ is an almost-polynomial of degree $\le s$ of $G$ that is separating, i.e., 
    $$\lim_{n\to\infty} d_{X_n}(f_n(g),id)=1$$
    for all $g\in G\setminus\{1\}$.
\end{definition}

\begin{remark}
    \label{rem:unitalsof}
    Recall that almost-homomorphisms are asymptotically unital by definition. Without assuming this, the notion of polynomial sofic approximations becomes degenerate. Namely, take any sequence of permutations $\sigma_n\in\Sym(X_n)$ with
    $$\lim_{n\to \infty }d_{X_n}(\sigma_n,1) =1$$
    and consider the constant maps
    $f_n(g)=\sigma_n$, $g\in G$. These are genuine polynomials of any degree, and
    $$\lim_{n\to\infty}d_{X_n}(f_n(g),1)=d_{X_n}(\sigma_n,1)=1.$$
    Thus every group would be admitting a polynomial sofic approximation of any degree under that definition.
\end{remark}

\begin{proof}[Proof of Theorem \ref{thm:soficcharacterization}]
    If $G$ is sofic, it admits a polynomial sofic approximation of degree $\le 1$, which is a polynomial sofic approximation of degree $\le s$ for every $s\ge 1$.

    The converse is proved by induction on $s$. Let $s\ge 2$ and let $G$ admit a polynomial sofic approximation $f_n:G\to\Sym(X_n)$ of degree $\le s$. Pick an arbitrary nonprincipal ultrafilter $\mathcal{U}$, and form the metric ultraproduct
    $\mathfrak{S}_{\mathcal{U}}$.
    Then $(f_n)_n$ induces a unital polynomial $f:G\to\mathfrak{S}_{\mathcal{U}}$ of degree $\le s$.
    Lemma \ref{lem:lowerdeg} implies that the map $f\mid_{G'}$ is a polynomial of degree $\le s-1$, hence it induces a polynomial sofic approximation of $G'$ of degree $s-1$. By induction, $G'$ is sofic, therefore $G$ is sofic.
\end{proof}

Theorem~\ref{thm:soficcharacterization} concerns
polynomial approximations of $G$ itself. A related,
but different, question is whether the universal
polynomial group $\Pold_s(G)$ is sofic.
The universal property gives an immediate
implication in one direction, which we record
for later use.

\begin{corollary}
    \label{cor:soficG}
    Let $G$ be a countable group. If $\Pold_s(G)$ is sofic, then $G$ is sofic.
\end{corollary}

\begin{proof}
    Let $\mathcal{U}$ be a nonprincipal ultrafilter admitting a homomorphism $f:\Pold_s(G)\to\mathfrak{S}_{\mathcal{U}}$ with $d_{\mathcal{U}}(f(p),1)=1$ for every $p\in\Pold_s(G)\setminus\{1\}$. Let $\phi=f\circ\iota_s:G\to\mathfrak{S}_{\mathcal{U}}$. Then $\phi$ is a unital polynomial of degree $\le s$. As $\iota_s$ is injective, we have that $d_{\mathcal{U}}(\phi(g),1)=1$ for every $g\in\setminus\{1\}$. This implies that $\phi$ induces a polynomial sofic approximation of $G$ of degree $\le s$. By Theorem \ref{thm:soficcharacterization}, $G$ is sofic.
\end{proof}

Although polynomial sofic approximations do not
enlarge the class of sofic groups, they need not
behave like ordinary sofic approximations.
In particular, an almost-polynomial of degree
two need not be an almost-homomorphism.
The following elementary example illustrates
this distinction.

\begin{example}[Polynomial sofic approximation of degree 2 for $\mathbb{Z}$]
    \label{ex:polysofZ}
    Let $p_n$ be a sequence of strictly increasing odd primes. Let $X_n=C_{p_n}\times C_{p_n}$ and $m_n=|X_n|=p_n^2$. For $x,y\in X_n$ let $a_n(x,y)=(x+1,y)$ and $b_n(x,y)=(x,y+1)$. The map $q_n:\mathbb{Z}\to\Sym(X_n)$ given by 
    $$q_n(k)=a_n^kb_n^{{k\choose 2}}$$
    is a unital quadratic map. Let $\tau_n\in\Sym(X_n)$ be the transposition interchanging $(0,0)$ and $(0,1)$. Define $f_n:\mathbb{Z}\to\Sym(X_n)$ by the rule
    $$f_n(k)=\left\{ \begin{array}{ccc}
        q_n(k) & : & k\neq 2\\
        \tau_n\cdot q_n(2) & : & k=2
    \end{array}
    \right . .$$
    For a fixed $k\in\mathbb{Z}$ we have that
    $$d_{X_n}(f_n(k),q_n(k))\le \frac{2}{p_n^2}\rightarrow 0.$$
    As the Hamming distance is bi-invariant, we therefore get
    $$d_{X_n}((\mypartial_{k_1,k_2,k_3}f_n)(1), 1)=
    d_{X_n}((\mypartial_{k_1,k_2,k_3}f_n)(1),\mypartial_{k_1,k_2,k_3}q_n)(1))\rightarrow 0$$
    for all $k_1,k_2,k_3\in\mathbb{Z}$. This shows that $(f_n)_n$ is an almost-polynomial of $\mathbb{Z}$ of degree exactly 2. Note that $f_n$ are not quadratic maps as $(\mypartial_{1,1,1}f_n)(0)\neq 1$. Fix a non-zero integer $k$. For sufficiently large $n$ we have that $p_n>|k|$, so $k\not\equiv 0\mod p_n$, therefore $q_n(k)$ is a nontrivial translation on $X_n$. Therefore we have that $d_{X_n}(q_n(k),1)=1$ for all $n$ large enough. In this case,
    $$d_{X_n}(f_n(k),1)\ge d_{X_n}(q_n(k),1)-d_{X_n}(f_n(k),q_n(k))\ge 1-\frac{2}{p_n^2}\rightarrow 1.$$
    This shows that $(f_n)$ is a sofic approximation of degree $2$ for $\mathbb{Z}$.
\end{example}


\section{Polynomial permutation stability}
\label{s:polystability}

\noindent

\begin{definition}[Close to polynomial]
    A sequence of maps $f_n:G\to\Sym(X_n)$ is \emph{close to a polynomial of degree $\le s$} if there exists a sequence of unital polynomials $\rho_n:G\to \Sym(X_n)$ of degree $\le s$ such that
    $$\lim_{n\to \infty} d_{X_n}(\rho_n(g),f_n(g))=0$$
for all $g\in G$.
\end{definition}

\begin{definition}[Permutation stable of degree $s$]
    A group $G$ is \emph{permutation stable of degree $\le s$} (\emph{$P_s$-stable} for short) if every almost-polynomial of degree $\le s$ of $G$ is close to a polynomial of degree $\le s$.
\end{definition}

Note that finite perfect groups are $P_s$-stable for every $s\ge 1$ by Corollary \ref{cor:perfect} and the usual $P$-stability of finite groups \cite{GR09}. 

\begin{proof}[Proof of Theorem \ref{thm:Pscharacterization}]
    We assume the notations set in Subsection \ref{ss:universal_prelim}.

    Let $G$ be a $P_s$-stable group. Let $f_n:\Pold_s(G)\to\Sym(X_n)$ be an almost-homomorphism of $\Pold_s(G)$. Define $\phi_n:G\to\Sym(X_n)$ by the rule $\phi_n=f_n\circ \iota_s$. As $\iota_s$ is unital and $(f_n)_n$ is an almost-homomorphism, it follows that $(\phi_n)_n$ is asymptotically unital. Fix $\mathbf{g}=(g_1,\ldots ,g_{s+1})\in G^{s+1}$. Then
    \begin{align*}
    d_{X_n}((\mypartial_{\mathbf{g}}\phi_n)(1),1) &\le
    d_{X_n}((\mypartial_{\mathbf{g}}(f_n\circ \iota_s))(1),
    f_n((\mypartial_{\mathbf{g}}\iota_s)(1)))\\
    &\quad + d_{X_n}(f_n((\mypartial_{\mathbf{g}}\iota_s)(1)),1). 
    \end{align*}
    As $\iota_s$ is a polynomial of degree $\le s$ and $(f_n)_n$ is an almost-homomorphism, we have that
    the latter sum converges to 0, therefore $(\phi_n)_n$ is an almost-polynomial of degree $\le s$ of $G$. By the $P_s$-stability of $G$ there exist unital polynomials $\rho_n:G\to\Sym(X_n)$ of degree $\le s$ with
    $$\lim_{n\to\infty}d_{X_n}(\phi_n(g),\rho_n(g))=0$$
    for every $g\in G$. By Proposition \ref{prop:Pold}, there exist (uniquely determined) homomorphisms $\psi_n:\Pold_s(G)\to\Sym(X_n)$ with $\rho_n=\psi_n\circ\iota_s$. Take $a\in\Pold_s(G)$ and write it as
$$a=\prod_{i=1}^k\iota_s(g_i)^{\varepsilon_i},$$
where $g_i\in G$. and $\varepsilon_i=\pm 1$. Then
$$\psi_n(a)=\prod_{i=1}^k\psi_n(\iota_s(g_i))^{\varepsilon_i}=\prod_{i=1}^{k}\rho_n(g_i)^{\varepsilon_i}.$$
On the other hand,
$$d_{X_n}\left (f_n(a),\prod_{i=1}^k\phi_n(g_i)^{\varepsilon_i}\right )=d_{X_n}\left (f_n\left (\prod_{i=1}^{k}\iota_s(g_i)^{\varepsilon_i}\right ), \prod_{i=1}^kf_n(\iota_s(g_i))^{\varepsilon_i}\right )$$
converges to 0, hence
\begin{align*}
    d_{X_n}(\psi_n(a),f_n(a)) &\le d_{X_n}\left ( \prod_{i=1}^{k}\rho_n(g_i)^{\varepsilon_i}, \prod_{i=1}^k\phi_n(g_i)^{\varepsilon_i}\right )+
    d_{X_n}\left (f_n(a),\prod_{i=1}^k\phi_n(g_i)^{\varepsilon_i}\right )\\
    &\le \sum_{i=1}^k d_{X_n}(\rho_n(g_i),\phi_n(g_i))+d_{X_n}\left (f_n(a),\prod_{i=1}^k\phi_n(g_i)^{\varepsilon_i}\right )
\end{align*}
converges to 0.

Conversely, assume that $\Pold_s(G)$ is $P$-stable. Let the sequence $f_n:G\to\Sym(X_n)$ be an almost-polynomial of $G$ of degree $\le s$. 
Replacing $f_n(g)$ by $f_n(g)f_n(1)^{-1}$, we may assume that every $f_n$ is unital. This replacement preserves all differences and changes the maps by a pointwise asymptotically vanishing Hamming distance.
Let $\widetilde{f}_n:F_G\to\Sym(X_n)$ be the unique homomorphism mapping $x_g$ to $f_n(g)$ for all $g\in G$. By the assumption on $f_n$ we have that
$$\lim_{n\to\infty} d_{X_n}(\widetilde{f}_n((\mypartial_{g_1,\ldots ,g_{s+1}}\iota)(1)),1)=0$$
for all $g_1,\ldots ,g_{s+1}\in G$. By Lemma \ref{lem:Pstabrel}, there exist homomorphisms $\rho_n:\Pold_s(G)\to\Sym(X_n)$ such that
$$\lim_{n\to\infty}d_{X_n}(\widetilde{f}_n(w),\rho_n(wR_s(G)))=0$$
for every $w\in F$. The maps $\phi_n:G\to\Sym(X_n)$, $\phi_n=\rho_n\circ\iota_s$, are unital polynomials of degree $\le s$, and
$$d_{X_n}(f_n(g),\phi_n(g))=d_{X_n}(\widetilde{f}_n(x_g),\rho_n(\iota_s(g)))=
d_{X_n}(\widetilde{f}_n(x_g),\rho_n(x_gR_s(G)))$$
converges to zero for every $g\in G$. This proves the claim.
\end{proof}

Theorem \ref{thm:Pscharacterization} implies that the cyclic group of order 2 is $P_s$ stable for every $s\ge 1$, as $\Pold_s(\Z/2\Z)\cong \Z/2^s\Z$ by \cite{JT24}. We also have the following:

\begin{corollary}
    \label{cor:finiteP2}
    Finite groups are $P_2$-stable.
\end{corollary}

\begin{proof}
    If $G$ is finite, then $\Pold_2(G)$ is finite by \cite[Corollary 4.4]{JT24}, hence $\Pold_2(G)$ is $P$-stable by \cite{GR09}.
\end{proof}

The preceding corollary provides a large class of groups
that are polynomially permutation stable in degree two.
For higher degrees, the situation is more subtle.
Finite perfect groups are $P_s$-stable for every $s\geq1$,
since polynomial maps on perfect groups are homomorphisms.
However, this observation does not yield examples
with genuinely nonlinear polynomial maps.

It is therefore natural to seek further classes of
$P_s$-stable groups, particularly among infinite
nonperfect groups. We shall show that polynomial
permutation stability can, under suitable assumptions,
be transferred between a group and one of its quotients.
The key observation is that the absence of nontrivial
homomorphisms into a target group forces all unital
polynomial maps into that target to be trivial.

\begin{lemma}
\label{lem:no-hom}
Let $K$ and $H$ be any groups such that $\Hom(K,H)=\{1\}$.
\begin{enumerate}
    \item Every unital polynomial $p:K\to H$  is trivial.
\item Suppose $K$ is a normal subgroup of $G$. Let $f:G\to H$ be a unital polynomial. Then $f$ factors through a unital polynomial $G/K\to H$ of the same degree bound.
\end{enumerate}
\end{lemma}

\begin{proof}
(1) Induction on the degree. A degree-zero map is the constant
map. Suppose the conclusion is known through degree $s-1$, and $p$
has degree at most $s$. For each $g\in K$, the map
$p_g(x)=p(gx)p(x)^{-1}p(g)^{-1}$
has degree at most $s-1$ and satisfies $p_g(1)=1$.
By induction $p_g=1$. Hence $p(gx)=p(g)p(x)$ for all $g,x\in K$,
so $p$ is a homomorphism. The hypothesis makes it trivial.

(2) Fix $g\in G$ and define $p_g:K\to H$ by $p_g(x)=f(gx)f(x)^{-1}f(g)^{-1}$. As in (1) we conclude that $p_g=1$, hence $f(gx)=f(g)f(x)$ for all $g\in G$, $x\in K$. Therefore $f=\bar f\circ\pi$ for a map $\bar f:G/K\to H$. For lifts $g_i,x\in G$,
$$
(\mypartial_{\pi(g_1),\ldots,\pi(g_r)}\bar f)(\pi(x))
=
(\mypartial_{g_1,\ldots,g_r}f)(x),
$$
which proves that $\bar f$ is polynomial with the required degree bound.
\end{proof}

\begin{proposition}
\label{thm:transfer}
Let $G$ be countable and $K$ a normal subgroup of $G$. Assume that $K$
has no nontrivial sofic quotient. Then, for every fixed $s\geq1$,
$G$ is $P_s$-stable if and only if $G/K$ is $P_s$-stable.
\end{proposition}

\begin{proof}
    Write $Q=G/K$ and $\pi:G\twoheadrightarrow Q$.

    Assume first that $Q$ is $P_s$-stable. Let the sequence of maps $f_n:G\to\Sym(X_n)$ be an almost-polynomial  of degree
at most $s$. Pick a nonprincipal ultrafilter $\mathcal{U}$. Then $(f_n)_n$ induces a unital polynomial $f:G\longrightarrow \mathfrak{S}_{\mathcal{U}}$ of degree $\le s$. By the assumption on $K$ we have that $\Hom(K,\mathfrak{S}_{\mathcal{U}})=\{1\}$. Namely, if $\varphi:K\longrightarrow \mathfrak{S}_{\mathcal{U}}$ is a homomorphism, then $\varphi(K)$ is a countable subgroup of $\mathfrak{S}_{\mathcal{U}}$ and is hence sofic by \cite[Theorem 13.2.11]{Gol22}. By Lemma \ref{lem:no-hom}, $f$ factors through $Q$. Consequently, for fixed $k\in K$ and $g\in G$,
$$
 \lim_{n\to\mathcal{U}}d_{X_n}(f_n(kg),f_n(g))=0.
$$
Since this holds for every $\mathcal{U}$, we conclude that
 $$\lim_{n\to\infty} d_{X_n}(f_n(kg),f_n(g)) = 0$$
for all $k\in K$, $g\in G$. Choose a section $t:Q\to G$ with $t(1)=1$, and put
$$
 a_n(q)=f_n(t(q)).
$$
The maps $a_n$ induce the map $f\circ t:Q\longrightarrow\mathfrak{S}_{\mathfrak{U}}$. As $f$ factors through $Q$, it follows that $f\circ t$ is exactly the descended polynomial. This gives that $(a_n)_n$ is an almost-polynomial of $Q$ of degree $\le s$. As $Q$ is $P_s$-stable, there exist unital polynomials $b_n:Q\to\Sym (X_n)$ of degree $\le s$ with $d_{X_n}(a_n(q),b_n(q))\to 0$ for each $q\in Q$. The maps $b_n\pi:G\to\Sym(X_n)$ are unital polynomials of degree $\le s$. Fix $g\in G$ and write it as $g=k t(\pi(g))$, where $k\in K$. Then
\begin{align*}
 d_{X_n}(f_n(g),b_n(\pi(g)))
 &\leq d_{X_n}(f_n(g),f_n(t(\pi(g))))\\
 &\quad+d_{X_n}(a_n(\pi(g)),b_n(\pi(g)))\longrightarrow0.
\end{align*}
This proves $P_s$-stability of $G$.

Conversely, assume that $G$ is $P_s$-stable. Let $a_n:Q\to\Sym(X_n)$ be almost-polynomial of degree at most $s$.
Then $(a_n\circ\pi )_n$ is almost-polynomial on $G$. Thus there exist unital polynomials $b_n:G\to\Sym(X_n)$ of degree $\le s$ that are close to $a_n\circ \pi$. As $\Hom(K,\Sym(X_n))=\{1\}$, Lemma \ref{lem:no-hom} implies that each $b_n$ factors through a unital polynomial $f_n:Q\to\Sym(X_n)$ of degree $\le s$. Choose $q\in Q$. Then
$$
 d_{X_n}(a_n(q),f_n(q))
 =d_{X_n}((a_n\circ\pi)(t(q)),b_n(t(q)))\longrightarrow0.
$$
Hence \(Q\) is $P_s$-stable.
\end{proof}

\begin{corollary}
\label{cor:examples}
Let $N$ be a countable simple nonsofic group. Both
$N*\Z/2\Z$ and $N\times \Z/2\Z$
are $P_s$-stable for
$s\geq1$. 
\end{corollary}

\begin{proof}
For the direct product, take $K=N\times\{\bar0\}$.
The quotient is $\Z/2\Z$, which is $P_s$-stable. The group $K$ clearly satisfies the  hypothesis of Proposition \ref{thm:transfer}, hence the result follows.

For the free product, let
$\pi:N*\Z/2\Z\to \Z/2\Z$
kill $N$ and be the identity on the other factor. Its kernel is
$K=\ncl{N}_{N*\Z/2\Z}$. This kernel is generated by the conjugates of
$N$ in $N*\Z/2\Z$. If $\theta:K\to S$ is a homomorphism to a sofic
group, its restriction to each such conjugate is trivial. Hence $\theta=1$. $P_s$-stability of $\Z/2\Z$
and Proposition \ref{thm:transfer} again apply.
\end{proof}

Specific examples can be found among  Brin--Thompson groups $nV$ \cite{Brin}. These are simple \cite{Brin}, finitely presented \cite{FMWZ}, and non-sofic \cite{SSFF} for $n$ large enough. Corollary \ref{cor:examples} implies that, for $n$ large enough, $nV\times \Z/2\Z$ and $nV* \Z/2\Z$ are $P_s$-stable for every $s\ge 1$.


\section{Cubical permutation instability of the cyclic group of order six}
\label{s:C6}

\noindent
We have seen that every finite group is
$P_2$-stable. We now show that this phenomenon
does not persist in degree three, even for
finite cyclic groups. Our aim is to prove
Theorem~\ref{thm:C6}, asserting that
$\Z/6\Z$ is not $P_3$-stable.
By Theorem~\ref{thm:Pscharacterization},
it suffices to establish that
$\Pold_3(\Z/6\Z)$ is not $P$-stable.
We shall prove the stronger assertion that
this universal group is not very flexibly
$P$-stable, allowing us to use Ioana's
instability criterion \cite{Ioa20} for direct products.

A countable group $\Gamma$ is very flexibly $P$-stable if every
almost-homomorphism $f_n:\Gamma\to\Sym(X_n)$ admits homomorphisms
$\rho_n:\Gamma\to\Sym(Y_n)$, with $X_n\subseteq Y_n$ finite, such that
$$
\frac{1}{|X_n|}
\bigl|\{x\in X_n:
f_n(g)(x)\ne\rho_n(g)(x)\}\bigr|
\longrightarrow 0
\qquad(g\in\Gamma).
$$
Ordinary $P$-stability implies very flexible $P$-stability by taking
$Y_n=X_n$. By \cite[Lemma 3.3]{Ioa20}, if a countable group is very flexibly $P$-stable, then so is every subgroup of finite index.

A crucial criterion for very flexible $P$-instability is the following result of Ioana: 

\begin{lemma}[Ioana~\cite{Ioa20}, Theorem~A]\label{lem:ioana}
Let $\Sigma$ and $\Lambda$ be finitely generated groups. If $\Sigma$ admits
a nonabelian free quotient and $\Lambda$ does not have property $(\tau)$,
then $\Sigma\times\Lambda$ is not very flexibly $P$-stable.
\end{lemma}

The preceding criterion suggests a strategy:
we seek a finite-index subgroup of
$\Pold_3(\Z/6\Z)$ that decomposes as a direct
product, with one factor admitting a nonabelian
free quotient and the other failing property
$(\tau)$. To carry out this construction,
we first obtain a concrete description of a
finite-index subgroup of the universal polynomial
group of a finite cyclic group.

The commutator of group elements $x$ and $y$ is defined by $[x,y]=xyx^{-1}y^{-1}$. For $n,e\ge 2$ define
\begin{equation}\label{eq:K}
\begin{split}
 K(n,e)=\langle x_0,\ldots,x_{n-1}\mid {}& x_i^e=1,\\
 &[[x_i,x_j],x_i]=[[x_i,x_j],x_j]=1\quad(i<j)\rangle.
\end{split}
\end{equation}
Define the epimorphism $\varepsilon_{n,e}:K(n,e)\to \Z/e\Z$ by sending $x_i$ to $\bar1$. Denote $L(n,e)=\ker\varepsilon_{n,e}$. Note that the groups $L(n,e)$ are finitely generated.

\begin{lemma}\label{lem:fourth}
Let $r^m=s^m=1$ in a group $E$, where $m\ge2$, and define
$f:\Z/m\Z\to E$ by $f(\bar i)=s^i r^{-i}$. For integers $i,k,h,j,\ell$, put
$t_k=r^{-k}s r^k$. Then
\begin{equation}\label{eq:fourth}
 (\mypartial_{\bar\ell ,\bar j,\bar h,\bar k} f)(\bar i)
 =s^{i+k+h}[s^\ell,[s^j,t_k^{-h}]]s^{-(i+k+h)}.
\end{equation}
Consequently $f$ is cubic if and only if every pair among
$t_k$, $k\in \Z$, generates a group of class at most two.
\end{lemma}

\begin{proof}
    Note that the rule for $f$ is independent of the integer representative of $\bar i$.
    Formula \eqref{eq:fourth} is proved by straightforward calculation.

    If every indicated pair generates a class-two subgroup, then clearly $f$ is cubic by \eqref{eq:fourth}, as all commutators $[s^\ell,[s^j,t_k^{-h}]]$ are trivial. Conversely, suppose that $f$ is cubic. Then the equation \eqref{eq:fourth} implies, in particular, that $[s,[s,t_k]]=1$ for all $k\in\Z$. Use this relation with $-k$ in place of $k$ and conjugate by $r^{-k}$ to obtain $[t_k,[t_k,s]]=1$ for all $k$. This gives that $\langle s,t_k\rangle$ has class $\le 2$. Conjugating by all powers of $r$ shows the same for every pair of the conjugates of $s$.
\end{proof}

\begin{proposition}
    \label{prop:C6semidirect}
    For every $m\ge 2$, put $P_m=\Pold_3(\Z/m\Z)$. There is an isomorphism
    \begin{equation}\label{eq:extension}
 P_m\rtimes \Z/m\Z\cong K(m,m)\rtimes \Z/m\Z,
\end{equation}
where the action on the left is the normalized translation action defined
below, and the action on the right cyclically permutes the $m$
distinguished generators of $K(m,m)$. 

Furthermore, the group $P_m$ has a normal subgroup of index $m$ isomorphic
to $L(m,m)$.
\end{proposition}

\begin{proof}
    Let $\iota:\Z/m\Z\longrightarrow P_m$ be the universal unital cubic map. Fix $k\in\Z$. The map
    $\bar i\mapsto \iota(\bar k)^{-1}\iota(\bar k+\bar i)$
    is unital and cubic, hence it induces an endomorphism $S_k$ of $P_m$ characterized by
    $$S_k(\iota(\bar i))=\iota(\bar k)^{-1}\iota(\bar k+\bar i).$$
    On the generators of $P_m$ we compute
\begin{align*}
 S_kS_\ell(\iota(\bar i))
 &=\bigl(\iota (\bar k)^{-1}\iota(\bar k+\bar \ell)\bigr)^{-1}
   \bigl(\iota(\bar k)^{-1}\iota(\bar k+\bar\ell+\bar i)\bigr)\\
 &=\iota(\bar k+\bar \ell)^{-1}\iota(\bar k+\bar\ell+\bar i)\\
 &=S_{k+ \ell}(\iota(\bar i)).
\end{align*}
Also $S_0=S_m=\id$. Hence these are automorphisms and define an action of
$\Z/m\Z$ upon $P_m$. Form the semidirect product $E=P_m\rtimes \Z/m\Z$ with respect to this action. Set $r=(1,\bar 1)$ and $s=(\iota(\bar 1),\bar1)=\iota(\bar1)r$. The identity $\iota(\bar i)S_i(\iota(\bar j))=\iota(\bar i+\bar j)$ gives, by induction $s^i=\iota(\bar i)r^i$. In particular, $r^m=s^m=1$, $\iota(\bar i)=s^ir^{-i}$ and $E=\langle r,s\rangle$. As $\iota$ is cubic, Lemma \ref{lem:fourth} implies that the elements $r^isr^{-i}$ satisfy the defining relations of $K(m,m)$.

Denote $E_*=K(m,m)\rtimes \langle r_*\rangle$ where $r_*^m=1$ and $r_*x_ir_*^{-1}=x_{i+1}$ with indices modulo $m$. Put $s_*=x_0$. The above argument shows that the assignment $r_*\mapsto r$, $x_i\mapsto r^isr^{-i}$ defines a homomorphism $\Psi:E_*\longrightarrow E$. To construct its inverse, note that Lemma \ref{lem:fourth} implies that the map $f_*(\bar i)=s_*^ir_*^{-i}$ is unital cubic. By the universal property, there is a homomorphism $\phi:P_m\rightarrow E_*$ with $\phi(\iota(\bar i))=f_*(\bar i)$. Moreover,
\begin{align*}
 \phi(S_1(\iota (\bar i)))
 &=f_*(\bar1)^{-1}f_*(\bar i+\bar 1)\\
 &=(s_*r_*^{-1})^{-1}s_*^{i+1}r_*^{-(i+1)}\\
 &=r_*s_*^i r_*^{-(i+1)}\\
 &=r_*f_*(\bar i)r_*^{-1}.
\end{align*}
Therefore $\phi$ extends to a homomorphism $\Phi:E\to E_*$ sending $r$ to
$r_*$. It sends $s=\iota(\bar 1)r$ to $s_*$. 
Both composites $\Psi\Phi$ and $\Phi\Psi$ fix the two generators 
$r,s$, respectively $r_*,s_*$, and are thus identities.

To prove the second part, note that the identity map of $\Z/m\Z$ induces an epimorphism $\kappa:P_m\rightarrow \Z/m\Z$ satisfying $\kappa(\iota(\bar i))=\bar i$. Define the maps $\pi_*$ and $\chi_*$ from $E_*$ to $\Z/m\Z$ by the rules
\begin{align*}
    \pi_*(kr_*^j) &=\bar j,\\
    \chi_*(kr_*^j) &=\epsilon_{m,m}(k)+\bar j.
\end{align*}
The map $\pi_*$ is the semidirect-product projection, so
$\ker\pi_*=K(m,m)$.
The cyclic permutation of the generators of $K(m,m)$ preserves $\varepsilon_{m,m}$, so $\chi_*$ is a homomorphism., and 
$$\chi_*(\phi(\iota(\bar i)))=\chi_*(s_*^ir_*^{-i})=0.$$
As $\iota(\bar i)$ generate $P_m$, it follows that $\chi_*\phi$ is trivial on $P_m$. This implies that $\chi_*(\Phi(p,\bar j))=\bar j$ for all $j\in\Z$ and all $p\in P_m$. As $\Phi$ is an isomorphism, it follows that $\ker\chi_*=\phi(P_m)$. Furthermore, 
$\chi_*|_{K(m,m)}=\varepsilon_{m,m}$. Therefore
$$
 \phi(P_m)\cap K(m,m)
 =
 \ker\chi_*\cap K(m,m)
 =
 \ker\varepsilon_{m,m}
 =
 L(m,m).
$$
Finally,
$$
 (\pi_*\circ\phi)(\iota(\bar i))
 =
 \pi_*(s_*^i r_*^{-i})
 =
 -i
 =
 -\kappa(\iota(\bar i)).
$$
Again, as the $\iota(\bar i)$ generate $P_m$, we have 
 $\pi_*\circ\phi=-\kappa$.
Taking kernels and using the injectivity of $\phi$ gives
$$
 \phi(\ker\kappa)
 =
 \phi(P_m)\cap\ker\pi_*
 =
 \phi(P_m)\cap K(m,m)
 =
 L(m,m).
$$
Thus $\phi|_{\ker\kappa}$ is an isomorphism between $\ker\kappa$ and $L(m,m)$.
Surjectivity of
$\kappa$ gives
$|P_m:\ker\kappa|=m$.
This concludes the proof.
\end{proof}

Proposition~\ref{prop:C6semidirect} reduces
the study of $\Pold_3(\Z/6\Z)$ to the group
$L(6,6)$. The next step is to exploit the
fact that
the defining relations of $K(n,6)$ allow
the contributions of orders two and three
to be separated, yielding a direct-product
decomposition of the corresponding kernels.

\begin{proposition}
    \label{prop:CRT}
For every $n\ge2$,
$ K(n,6)\cong K(n,2)\times K(n,3)$.
Under this isomorphism,
$L(n,6)\cong L(n,2)\times L(n,3)$.
\end{proposition}

\begin{proof}
    In $K(n,6)$ put $a_i=x_i^3$ and $b_i=x_i^4$. Then
$a_i^2=b_i^3=1$ and $a_i b_i=x_i^7=x_i$.
If $c_{ij}=[x_i,x_j]$, the defining relations say that $c_{ij}$
centralizes $x_i$ and $x_j$. The class-two commutator identity gives
 $c_{ij}^6=[x_i^6,x_j]=1$ and
$[a_i,b_j]=[x_i^3,x_j^4]=c_{ij}^{12}=1$ for all $i$ and $j$.
Every pair of $a_i$'s, and
every pair of $b_i$'s, lies in a class-two subgroup. 

Let $u_0,\ldots,u_{n-1}$ and $v_0,\ldots,v_{n-1}$ denote the
distinguished generators of $K(n,2)$ and $K(n,3)$, respectively.
The assignments
$\alpha:K(n,2)\longrightarrow K(n,6)$,
$\alpha(u_i)=a_i=x_i^3$,
and
 $\beta:K(n,3)\longrightarrow K(n,6)$,
$\beta(v_i)=b_i=x_i^4$
define homomorphisms.
This yields a homomorphism
$$
 \Phi:K(n,2)\times K(n,3)\longrightarrow K(n,6)
$$
given by $\Phi(x,y)=\alpha(x)\beta(y)$.
It is onto since $x_i=a_i b_i$.

The inverse of $\Phi$ is given by the rule $\Psi(x_i)= (u_i,v_i)$; it is clearly well-defined, $\Phi\Psi(x_i)=x_i^7=x_i$, $\Psi\Phi(u_i,1)=(u_i^3,v_i^3)=(u_i,1)$, and
$\Psi\Phi(1,v_i)=(u_i^4,v_i^4)=(1,v_i)$.

Note that $\varepsilon_{n,6}(a_i)=\bar3$ and $\varepsilon_{n,6}(b_i)=\bar 4$. If $\theta:\Z/6\Z\longrightarrow \Z/2\Z\times \Z/3\Z$ is the standard isomorphism, then, by looking at the images of $(u_i,v_j)$, we see that $\varepsilon_{n,6}\circ \Phi=\theta^{-1}\circ (\varepsilon_{n,2}\times\varepsilon_{n,3})$. This proves that $\ker\varepsilon_{n,6}\cong \ker(\varepsilon_{n,2}\times \varepsilon_{n,3})$, hence the result.
\end{proof}

For the remainder of this section fix the following notations:
\begin{align*}
    P &= \Pold_3(\Z/6\Z),\\
    A &= L(6,2),\\
    B &= L(6,3).
\end{align*}
Proposition \ref{prop:C6semidirect} and Proposition \ref{prop:CRT} imply that $P$ has a normal subgroup $P_0$ isomorphic to $A\times B$ with $|P:P_0|=6$.

We now have the required direct-product
structure. It remains to find suitable
finite-index subgroups of its two factors.
For this purpose, it is enough to construct
appropriate quotients of $A$ and $B$.
The following elementary observation reduces
these questions to groups with only three
distinguished generators.

\begin{lemma}\label{lem:folding}
For $e=2,3$, sending the first three distinguished generators of
$K(6,e)$ to the three distinguished generators of $K(3,e)$, and each
remaining generator to the first one, induces an epimorphism
$\rho_e:K(6,e)\twoheadrightarrow K(3,e)$, which restricts to an epimorphism $L(6,e)\twoheadrightarrow L(3,e)$.
\end{lemma}

\begin{proof}
Each pair of images is either a pair of distinguished generators or a
repeated generator, so it satisfies all required class-two relations.
Orders are also respected, and the map is onto. 

If an element of $L(3,e)$ is lifted to
$K(6,e)$, its lift has the same exponent sum, namely zero, and therefore
already belongs to $L(6,e)$. This proves surjectivity on the kernels.
\end{proof}

\begin{proposition}\label{prop:A}
There is a subgroup $A'\le A$ with $|A:A'|=16$ and an epimorphism
$A' \twoheadrightarrow F_3$.
\end{proposition}

\begin{proof}
    The group $K(3,2)$ is generated by three involutions, $a$, $b$ and $c$. We have that $[a,b]=(ab)^2$. Centrality of $[a,b]$ in $\langle a,b\rangle$ is equivalent to $(ab)^4=1$, hence
    $$K(3,2)=\langle a,b,c\mid a^2=b^2=c^2=(ab)^4=(bc)^4=(ca)^4=1\rangle.$$
    Set $u=ab$,
    $v=bc$, so $a=ub$ and $c=bv$. Tietze transformations
    give
 \begin{align*}
    K(3,2) &=\langle u,v,b\mid u^4=v^4=(uv)^4=b^2=1,
 \ b u b=u^{-1},\ b v b=v^{-1}\rangle\\
 &=\Delta\rtimes \langle b\rangle\\
 &\cong \Delta\rtimes \Z/2\Z,
 \end{align*}
 where 
 $$ \Delta=\langle u,v\mid u^4=v^4=(uv)^4=1\rangle$$
 is the $(4,4,4)$ triangle group. From this presentation it is straightforward to see that $L(3,2)=\Delta$.
We have that
$\Delta^{\mathrm{ab}}\cong C_4\times C_4$. Let 
 $M$ be the commutator subgroup of $\Delta$.
Laing and Singerman~\cite[Section~9, p.~297]{LS12} identify
the commutator subgroup of the triangle group
$\Delta(n,n,n)$ as the surface group uniformizing the Fermat
curve of degree $n$, whose genus is $(n-1)(n-2)/2$.
Thus $M$ is the fundamental group of a closed orientable
surface of genus $3$, and $[\Delta:M]=16$. We have
$$M\cong\langle p_1,q_1,p_2,q_2,p_3,q_3\mid
 [p_1,q_1][p_2,q_2][p_3,q_3]=1\rangle .$$
Sending $q_1,q_2,q_3$ to $1$ and $p_1,p_2,p_3$ to a free basis gives
$M\twoheadrightarrow F_3$.

Finally, let $A'$ be the preimage of $M$ under $A\twoheadrightarrow\Delta$.
Because that map is surjective, $[A:A']=[\Delta:M]=16$, and its
restriction onto $M$ is surjective. Composing with $M\twoheadrightarrow F_3$ proves
the proposition.
\end{proof}

The first factor has now been dealt with:
the triangle-group description provides
a finite-index subgroup with a nonabelian
free quotient. For the second factor,
Ioana's criterion requires only the failure
of property $(\tau)$. It will therefore
suffice to find a finite-index subgroup
with an infinite cyclic quotient.
We achieve this by constructing an affine
representation over the Eisenstein integers.

\begin{proposition}\label{prop:B}
There is a subgroup $B'\le B$ with $|B:B'|=27$ and an epimorphism
$B'\twoheadrightarrow \Z$.
\end{proposition}

\begin{proof}
    Denote by $a$, $b$ and $c$ the distinguished generators of the group $T=K(3,3)$. Let $\epsilon=\varepsilon_{3,3}:T\twoheadrightarrow \Z/3\Z$ 
    and put $T_0=\ker\epsilon=L(3,3)$.
    By Lemma~\ref{lem:folding}, there is an epimorphism $B\twoheadrightarrow T_0$.
    We construct a subgroup of index $27$ in $T_0$ with an infinite cyclic
    quotient. Its preimage $B'$ in $B$ will have the desired properties.

    Let $\omega\in\C$ be a primitive cube root of unity and
$R=\Z[\omega]$. Put
\begin{equation}\label{eq:SD}
 S=\begin{pmatrix}0&0&1\\1&0&0\\0&1&0\end{pmatrix},\qquad
 D=\begin{pmatrix}1&0&0\\0&\omega&0\\0&0&\omega^2\end{pmatrix}.
\end{equation}
The group $H=\langle S,D\rangle$ has order 27, exponent 3 and class 2. 

Let $\operatorname{Aff}_3(R)=\operatorname{GL}_3(R)\ltimes R^3$
denote the group of invertible affine transformations of $R^3$.  Set
\begin{equation}\label{eq:Zt}
 Z=\omega^2SD=
 \begin{pmatrix}0&0&\omega\\\omega^2&0&0\\0&1&0\end{pmatrix},
 \qquad
 \mathbf t=\begin{pmatrix}1-\omega\\1-\omega^2\\0\end{pmatrix},
\end{equation}
and consider
the elements
 $\alpha=(S,0)$, $\beta=(D,0)$, and $\gamma=(Z,\mathbf t)$ of $\operatorname{Aff}_3(R)$. 
 Set
$$
p=(1,1,1)^{\mathsf T},
\qquad
q=(1-\omega,0,0)^{\mathsf T}.
$$
Direct multiplication gives
$Sp=p$, $Dq=q$, and
$(I-Z)p=(I-Z)q=\mathbf t$.
Writing $\tau_v=(I,v)$, it follows that
$\tau_p(S,0)\tau_p^{-1}=\alpha$,
$\tau_p(Z,0)\tau_p^{-1}=\gamma$,
and
$\tau_q(D,0)\tau_q^{-1}=\beta$,
$\tau_q(Z,0)\tau_q^{-1}=\gamma$.
Thus $\langle\alpha,\gamma\rangle$ and $\langle\beta,\gamma\rangle$ are conjugate to subgroups of the copy of $H$ inside $\operatorname{Aff}_3(R)$, and $\langle \alpha,\beta\rangle\cong H$. 
Hence
 we find that there is a homomorphism 
 $\theta:T\longrightarrow\operatorname{Aff}_3(R)$ given by
 $\theta(a)=\alpha$, $\theta(b)=\beta$, $\theta(c)=\gamma$. 
 
 Composing $\theta$ with the projection $\operatorname{Aff}_3(R)\twoheadrightarrow \operatorname{GL}_3(R)$, we obtain an epimorphism $\ell:T\twoheadrightarrow H$ with $\ell(a)=S$, $\ell(b)=D$, $\ell(c)=Z$.
Define $N=T_0\cap\ker\ell$.
The restriction $\ell|_{T_0}$ is onto $H$.
Indeed, let 
 $$w=c\bigl([a,b]ab\bigr)^{-1}.$$
 Then $\ell(w)=1$ and
$\epsilon(w)=\bar1-\bar1-\bar1=-\bar1$ in $\Z/3\Z$. It follows that
$w^3\in N$.
If $z\in T$ is any lift of an element of $H$, multiply $z$ by
$w^j$, where $j\equiv\epsilon(z)\pmod3$. This does not change its linear
image and makes its exponent sum zero. Consequently
\begin{equation}\label{eq:N-index}
 [T_0:N]=|H|=27.
\end{equation}
The restriction $\theta|_N$ takes values in the translation group
$(R^3,+)\cong\Z^6$. Composing it with the additive functional that takes
the coefficient of $1$ in the first coordinate gives a homomorphism
$\nu:N\to\Z$. 
We have $\theta(w)=(Z,\mathbf t)(Z^{-1},0)=(I,\mathbf t)$. As noted above, $w^3\in N$, and we get that
$\theta(w^3)=(I,3\mathbf t)\ne1$.
This gives $\nu(w^3)=3$.
Thus $\nu(N)$ is a nonzero subgroup $d\Z$, and division by its positive
generator $d$ gives an epimorphism $N\twoheadrightarrow\Z$.
\end{proof}

Note, in particular, that $B'$ does not have the property $(\tau)$, as 
every finitely generated group with property $(\tau)$
has finite abelianization \cite[Corollary~1.29]{LZ}.

\begin{proof}[Proof of Theorem \ref{thm:C6}]
    The group $P=\Pold_3(C_6)$ has a subgroup
$P_0\cong A\times B$ of index six. Propositions~\ref{prop:A}
and~\ref{prop:B} give subgroups $A'\le A$, $B'\le B$ with
$|A:A'|=16$, $A'\twoheadrightarrow F_3$,
and
$|B:B'|=27$, $B'\twoheadrightarrow\Z$.
 Therefore
$|P:A'\times B'|=6\cdot16\cdot27=2592$.
The groups $K(6,2)$ and $K(6,3)$ are finitely generated. The groups
$A,B$, and then $A',B'$, have finite index in these groups, so they too
are finitely generated. Apply Lemma \ref{lem:ioana} with
$\Sigma=A'$ and $\Lambda=B'$. It follows that $P$ is not very flexibly
$P$-stable and, in particular, not $P$-stable.
Theorem~\ref{thm:Pscharacterization} now implies that $\Z/6\Z$ is not
$P_3$-stable.
\end{proof}

\section{Universal polynomial groups}
\label{s:Pold}

\noindent
The preceding sections illustrate how the
structure of universal polynomial groups
governs the approximation properties of
polynomial maps. We now study several
structural properties of these groups in
their own right. We first consider soficity
and property $(T)$, where the quadratic
case admits a particularly useful description.
We then turn to nonamenability and largeness
in higher degrees.

\subsection{Soficity and property $(T)$}
\label{s:poldsofpropT}

\begin{proposition}
    \label{prop:Pol2embed}
    Let $G$ be a group. Then $\Pold_2(G)$ embeds in $G^{\rm ab}\wr G$, and there is a split exact sequence
    $$1\longrightarrow \Pold_2(G)\longrightarrow G^{\rm ab}\wr G\longrightarrow G^{\rm ab}\longrightarrow 1.$$
\end{proposition}

\begin{proof}
    It follows from \cite[Theorem 1.3]{JT24} that $\Pold_2(G)$ is isomorphic to the twisted cross product $A\rtimes_\psi G$, where $A=\omega(G)\otimes_\mathbb{Z} G^{\rm ab}$. Here $\omega(G)$ is the augmentation ideal of $\mathbb{Z}G$. Note that the embedding of $\omega(G)$ into $\mathbb{Z}G$ is split, hence tensoring with $G^{\rm ab}$ keeps it injective.
    The 2-cocycle $\psi:G\times G\to A$ is given by $\psi(g,h)=(g-1)\otimes \bar{h}$, where $\bar{h}=hG'$.

    Consider $M=\mathbb{Z}G\otimes_\mathbb{Z} G^{\rm ab}$. As a $G$-module,
    $$M\cong \bigoplus_{g\in G} G^{\rm ab},$$
    with $G$ acting by left translation. Thus $M\rtimes G\cong G^{\rm ab}\wr G$. Define $\Phi:A\rtimes_\psi G\to M\rtimes G$ by
    $$\Phi(\xi,g)=(\xi+1\otimes \bar{g},g)$$
    for all $\xi\in A, g\in G$. As
    \begin{align*}
        \Phi((\xi_1,g_1)(\xi_2,g_2)) &= 
        \Phi(\xi_1+g_1\xi_2+\psi(g_1,g_2),g_1g_2)\\
        &= (\xi_1+g_1\xi_2+\psi(g_1,g_2)+1\otimes (\bar{g}_1+\bar{g}_2),g_1g_2)\\
        &= (\xi_1+g_1\xi_2+g_1\otimes \bar{g}_2+1\otimes \bar{g}_1,g_1g_2)\\
        &=\Phi(\xi_1,g_1)\Phi(\xi_2,g_2),
    \end{align*}
    it follows that $\Phi$ is an injective homomorphism, hence $\Pold_2(G)$ embeds in $G^{\rm ab}\wr G$.

    In addition to the above,
$$
\Phi(\Pold_2(G))
=\{(\eta,g)\in M\rtimes G:\sigma(\eta)=\bar g\},
$$
where $\sigma=\varepsilon\otimes1:M\to G^{\mathrm{ab}}$.
Indeed, $\sigma(\xi+1\otimes\bar g)=\bar g$, and conversely
$\sigma(\eta)=\bar g$ implies
$\eta-1\otimes\bar g\in\ker\sigma
=\omega(G)\otimes G^{\mathrm{ab}}$.
Thus the image is the kernel of
$$
\chi:M\rtimes G\longrightarrow G^{\mathrm{ab}}$$
given by
$(\eta,g)\longmapsto\sigma(\eta)-\bar g$. Note that $\chi$ is indeed a homomorphism, since $\sigma(g\eta)=\sigma(\eta)$, and the multiplication in $M\rtimes G$ gives
$\chi((\eta,g)(\zeta,h))
=\chi(\eta,g)+\chi(\zeta,h)$.

Define
$j:G^{\rm ab}\longrightarrow M\rtimes G$ by
$j(a)=(1\otimes a,1_G)$.
This is a homomorphism, since
$j(a)j(b)=(1\otimes(a+b),1_G)=j(a+b)$,
and
$\chi(j(a))=\sigma(1\otimes a)-\overline{1_G}=a$.
Therefore
$\chi\circ j=\operatorname{id}_{G^{\rm ab}}$.
This proves both surjectivity of \(\chi\) and splitting of the sequence.
\end{proof}

\begin{corollary}
\label{cor:Pol2sofic}
If $G$ is sofic, then $\Pold_2(G)$ is sofic.
\end{corollary}

\begin{proof}
The group $G^{\rm ab}$ is abelian, hence sofic.
Therefore $G^{\rm ab}\wr G$ is sofic by
\cite[Theorem~1.1]{HS18}.
By Proposition~\ref{prop:Pol2embed},
$\Pold_2(G)$ embeds in $G^{\rm ab}\wr G$.
Since subgroups of sofic groups are sofic, the result follows.
\end{proof}

The wreath-product description also
has consequences for rigidity properties
of universal polynomial groups.
In contrast to soficity, property $(T)$
places strong restrictions on the
abelianization of the original group.
For infinite groups, this leads to
a complete characterization in every
degree at least two.

\begin{corollary}
\label{cor:PoldT}
Let $G$ be an infinite group and let $s\geq2$.
Then $\Pold_s(G)$ has property $(T)$ if and only if
$G$ is perfect and has property $(T)$.
\end{corollary}

\begin{proof}
Suppose that $\Pold_s(G)$ has property $(T)$.
The natural epimorphisms
\[
\Pold_s(G)\twoheadrightarrow\Pold_2(G)\twoheadrightarrow G
\]
show that both $\Pold_2(G)$ and $G$ have property $(T)$.
In particular, $G^{\rm ab}$ is finite.
By Proposition~\ref{prop:Pol2embed},
$\Pold_2(G)$ is isomorphic to a subgroup of index
$|G^{\rm ab}|$ in $G^{\rm ab}\wr G$.
Since property $(T)$ passes to finite-index supergroups,
$G^{\rm ab}\wr G$ has property $(T)$.
As $G$ is infinite, \cite[Theorem~1.2]{Neu05} implies
that $G^{\rm ab}=1$. Thus $G$ is perfect.

Conversely, if $G$ is perfect, then every unital polynomial
map on $G$ is a homomorphism by
\cite[Corollary~1.5]{JT24}.
The universal property therefore gives
$\Pold_s(G)\cong G$.
Hence $\Pold_s(G)$ has property $(T)$ whenever $G$ does.
\end{proof}

\subsection{$\Pold_s(G)$ and largeness}
\label{s:poldfinite}

\noindent
We assume throughout this section that $G$ is finite group. We first record some basic observations. Let $\alpha:G\to Q$ be a homomorphism and  $\iota_s^Q:Q\to\Pold_s(Q)$  and $\iota_s^G:G\to\Pold_s(G)$ the universal polynomials of degree $\le s$ on $Q$ and $G$, respectively. Then the universal property implies that there exists a uniquely determined homomorphism $\alpha_s^*:\Pold_s(G)\to\Pold_s(Q)$ with $\alpha_s^*\circ \iota_s^G=\iota_s\circ \alpha$. This is the functoriality mentioned in \cite{JT24}. Furthermore, if $\alpha$ is surjective, then the image of $\alpha_s^*$ contains $\iota_s^Q(Q)$, which generates $\Pold_s(Q)$, hence $\alpha_s^*$ is also surjective.

\begin{proposition}
    \label{prop:Gabdivisibleby3}
    If $|G^{\rm ab}|$ is divisible by $3$, then $\Pold_s(G)$ is non-amenable for every $s\ge 3$.
\end{proposition}

\begin{proof}
    By the assumption, $G$ has a quotient isomorphic to $C_3$. Note that $\Pold_3(C_3)$ is non-amenable. For every $s\ge 3$, we have an epimorphism
    $$\Pold_s(G)\twoheadrightarrow \Pold_s(C_3)\twoheadrightarrow \Pold_3(C_3),$$
    hence $\Pold_s(G)$ is non-amenable.
\end{proof}

The preceding proposition gives a simple
obstruction to amenability when the
abelianization of $G$ has order divisible
by three. It does not apply to groups
whose abelianization is a $2$-group.
We next consider a family with
$G^{\rm ab}\cong C_2$ and obtain a stronger
conclusion: their universal polynomial
groups are large in all sufficiently
high degrees.
A group $G$ is called \emph{large} if there exists a finite-index
subgroup $H\leq G$ and an epimorphism
$H\twoheadrightarrow F_2$
where $F_2$ denotes the free group of rank two. Note that large groups are clearly non-amenable.

\begin{proposition}
    \label{prop:c2perfect}
    There exist finite groups $G$ with $G'$ perfect, $G^{\rm ab}\cong C_2$, such that $\Pold_s(G)$ are large for all sufficiently large $s$.
\end{proposition}

\begin{proof}
    Let $N$ be a nontrivial finite perfect group of order $q$. Let $\alpha\in\Aut(N)$ be an involution. Form $G=N\rtimes_\alpha \langle t\rangle$, where $\langle t\rangle\cong C_2$. We claim that $\Pold_s(G)$ are large for sufficiently large $s$.

    Let $f:G\to H$ be a unital polynomial of degree $\le s$, and denote $c=f(t)$. Consider the maps $\mathfrak{a}:N\to H$ and $\mathfrak{b}:N\to H$ given by
    $\mathfrak{a}(n)=f(n)$ and $\mathfrak{b}(n)=f(nt)c^{-1}$. Clearly, $\mathfrak{a}$ and $\mathfrak{b}$ are unital polynomials of degree $\le s$ on $N$. As $N$ is perfect, they are homomorphisms by \cite[]{JT24}.

    Define $\beta_t:G\to H$ by $\beta_t(x)=c^{-1}f(tx)f(x)^{-1}$. This is a normalized left difference of $f$, hence it is a unital polynomial of degree $\le s-1$ on $G$. Consider the maps $\mathfrak{c}:N\to H$ and $\mathfrak{d}:N\to H$ given by $\mathfrak{c}(n)=\beta_t(n)$ and $\mathfrak{d}(n)=\beta_t(nt)\beta_t(t)^{-1}$. These are unital polynomials of degree $\le s-1$ on $N$, hence $\mathfrak{c}$ and $\mathfrak{d}$ are also homomorphisms.
    
    A direct calculation shows that the relation $\alpha(n)t=tn$ yields
    \begin{equation}
        \label{eq:Nt1}
        \mathfrak{c}(n)=c^{-1}\mathfrak{b}(\alpha(n))c\mathfrak{a}(n)^{-1}.
    \end{equation}
    We also note that $\beta_t(t)=c^{-1}f(t^2)f(t)^{-1}=c^{-2}$. This, together with straightforward calculation, shows that
    \begin{equation}
        \label{eq:Nt2}
        \mathfrak{d}(n)=c^{-1}\mathfrak{a}(\alpha(n))c^{-1}\mathfrak{b}(n)^{-1}c^2.
    \end{equation}
    We conclude that
    $$c\mathfrak{d}(\alpha(n))c^{-1}=\mathfrak{a}(\alpha^2(n))c^{-1}\mathfrak{b}(\alpha(n))^{-1}c=\mathfrak{c}(n)^{-1}.$$
    This shows that the map $n\mapsto\mathfrak{c}(n)^{-1}$ is a homomorphism $N\to H$. But $\mathfrak{c}$ is also a homomorphism, therefore we see that $\mathfrak{c}(N)$ needs to be abelian. But this is also a perfect group, hence $\mathfrak{c}(N)=1$, and therefore $\mathfrak{d}(N)=1$.

    The equation \eqref{eq:Nt1} yields $\mathfrak{b}(\alpha(n))=c\mathfrak{a}(n)c^{-1}$, thus $\mathfrak{b}(n)=c\mathfrak{a}(\alpha(n))c^{-1}$. In summary, we have that $f(n)=\mathfrak{a}(n)$ and $f(nt)=c\mathfrak{a}(\alpha(n))$ for all $n\in N$. This implies that 
    \begin{equation}
        \label{eq:Nt3}
        f(gn)=f(g)\mathfrak{a}(n)
    \end{equation}
    for all $g\in G$ and all $n\in N$.

    Given $g\in G$, consider the normalized difference map $\beta_g:G\to H$ given by $\beta_g(x)=f(g)^{-1}f(gx)f(x)^{-1}$. Given $x\in G$ and $n\in N$, we get from Equation \eqref{eq:Nt3} that $\beta_g(xn)=\beta_g(x)$. This shows that $\beta_g$ factors through $G/N\cong C_2$.

    We recall from \cite{JT24,AT26} that $\Pold_s(C_2)\cong C_{2^s}$. In other words, if $p:C_2\to H$ is a unital map sending the generator of $C_2$ to $u\in H$, then $p$ is a polynomial of degree $\le s$ if and only if $u^{2^s}=1$.

    Given $n\in N$, denote
    $$u_n=\beta_n(t)=\mathfrak{a}(n)^{-1}c\mathfrak{a}(\alpha(n))c^{-1}$$
    and
    $$v_n=\beta_{nt}(t)=\mathfrak{a}(\alpha(n))^{-1}c^{-1}\mathfrak{a}(n)c^{-1}.$$
    Note that $f$ is a polynomial of degree $\le s$ if and only $\beta_g$ are polynomials of degree $\le s-1$ for all $g\in G$. By the above, this is precisely when there exist a homomorphism $\mathfrak{a}:N\to H$, and an element $c\in H$ such that
    \begin{align*}
        f(n) &= \mathfrak{a}(n),\\
        f(nt) &= c\mathfrak{a}(\alpha(n)),\\
        u_n^M &= 1,\\
        v_n^M &=1,
    \end{align*}
    for all $n\in N$, where $u_n$ and $v_n$ are as above, and $M=2^{s-1}$. This immediately implies that
    $$\Pold_s(G)\cong \langle N,c\mid (n^{-1}c\alpha(n)c^{-1})^M=(\alpha(n)^{-1}c^{-1}nc^{-1})^M=1,\, n\in N\rangle.$$
    Note that the second set of relators implies $c^{2M}=1$.

   Let $\Gamma_0=N*\langle c\rangle \cong N*\mathbb Z,$
    where at this stage $c$ has infinite order.  The relation
    $c^{2^s}=1$ will be imposed below as one of the
    $2^{s-1}$-power relations. We have an epimorphism $\eta:\Gamma_0\to G$ sending $n\mapsto n$, $c\mapsto t$. Let $F$ be the kernel of $\eta$. As $\eta$ is injective on $N$, Kurosh's theorem implies that $F$ is free, and
    $|\Gamma_0:F|=|G|=2q$. By multiplicativity of the group Euler characteristic under passage to
    finite-index subgroups and the free-product formula
    \cite[Chapter IX, Proposition 7.3]{Bro82}, we have
    $$\chi(F)=[\Gamma_0:F]\chi(\Gamma_0)=2q\left(\frac1{q}-1\right)=2-2q.$$
    This shows that $F$ is free of rank $2q-1$, hence non-abelian. Note also that the "roots" of all relators of $\Pold_s(G)$ belong to $F$:
    \begin{align*}
        \eta(n^{-1}c\alpha(n)c^{-1}) &=n^{-1}t\alpha(n)t^{-1}=1,\\
        \eta(\alpha(n)^{-1}c^{-1}nc^{-1}) &=\alpha(n)^{-1}tnt=1.
    \end{align*}
    Choose a transversal $\mathcal{T}$ for $F$ in $\Gamma_0$. As $F$ is a normal subgroup of $\Gamma_0$, we conclude that the kernel of the epimorphism $\Pold_s(G)\twoheadrightarrow G$ has the form
    $$F_s=F/\llangle (\gamma r\gamma^{-1})^M\mid r\in\mathcal{R},\gamma\in\mathcal{T}\rrangle,$$
    where 
    $$\mathcal{R}=\{ u_n\mid n\in N, n\neq 1\} \cup\{v_n\mid n\in N\}$$
    is a fixed finite set independent of $s$, and the same holds for $\mathcal{T}$. By a result of Lackenby \cite[Theorem 1.2]{Lac07}, $F_s$ is large for sufficiently large $s$. As $F_s$ has finite index in $\Pold_s(G)$, it follows that the latter group is large for sufficiently large $s$. 
    \end{proof}

    The above result gives no particular lower bound for $s$. Keeping the notations of the proof of Proposition \ref{prop:c2perfect}, we have a presentation of $F_s$ with $r=2q-1$ generators and $\le 2q(2q-1)=2qr$ relators, each of which is a  $2^{s-1}$-power of an element in $F$. Hence the $2$-deficiency of $F_s$ is bounded from below as follows:
    $$\deff_2(F_s)\ge r-\frac{2qr}{2^{s-1}}.$$
    It is shown in \cite{BT11} that if the $p$-deficiency of a finitely presented group is bigger than 1, then the group is large. Thus the largeness of the group $F_s$, and therefore of $\Pold_s(G)$ is ensured by the condition
    $$2^{s-1}>\frac{2qr}{r-1}.$$
    For example, consider the group $\Sym(n)$, where $n\ge 5$. Then we can 
    take $N=\Alt(n)$, and $\Pold_s(\Sym(n))$ is large if
    $$s>1+\log_2\frac{n!(n!-1)}{n!-2}.$$

\section*{Use of AI tools}

\noindent
ChatGPT was used solely to assist with the manuscript’s structure, organization, formatting, and checking the arguments. All content was reviewed and substantially
revised by the author, who is responsible for the final text.


\end{document}